\documentclass[11pt]{article}
\usepackage{latexsym}

\usepackage{amssymb}
\usepackage{amsmath}
\usepackage{mathrsfs}
\usepackage{enumerate}

\newcommand{\wwt}{\operatorname{w}}
\newcommand{\vwt}{\operatorname{v}}

\newcommand{\tinyspace}{\vspace*{-2mm}}

\newcommand{\h}{\nu}

\font\Cp = msbm10

\newcommand{\Sym}{{\mathfrak S}}
\newcommand{\Rees}{\operatorname{Rees}}
\newcommand{\Cbasis}{{\mathscr C}}

\newcommand{\hz}{\hat{0}}
\newcommand{\ho}{\hat{1}}
\newcommand{\corank}{\operatorname{corank}}

\newcommand{\Fall}{{\cal F}}

\newcommand{\onethingatopanother}[2]{\genfrac{}{}{0pt}{}{#1}{#2}}

\newcommand{\signed}[1]{\Sym_{#1}^{\pm}}
\newcommand{\barredsigned}[1]{\overline{\Sym_{#1}^{\pm}}}

\newtheorem{theorem}{Theorem}[section]
\newtheorem{proposition}[theorem]{Proposition}
\newtheorem{lemma}[theorem]{Lemma}
\newtheorem{definition}[theorem]{Definition}
\newtheorem{corollary}[theorem]{Corollary}
\newtheorem{example}[theorem]{Example}
\newtheorem{claim}[theorem]{Claim}

\font\Cp = msbm10

\newcommand{\Ccc}{\hbox{\Cp C}}

\newcommand{\Ppp}{\hbox{\Cp P}}
\newcommand{\Qqq}{\hbox{\Cp Q}}

\newcommand{\Zzz}{\hbox{\Cp Z}}

\newcommand{\qed}{\mbox{$\Box$}\vspace{\baselineskip}}

\newenvironment{proof}{\noindent {\bf Proof:}}{{\qed}}
\newenvironment{proof_special}[1]{\noindent {\bf Proof of #1:}}{{\qed}}

\title{The Rees product of the cubical lattice and graded posets
with the chain}
\title{The Rees product of posets}
\author{\sc Patricia Muldoon Brown and 
Margaret A.\ Readdy\thanks{Part of this work 
was completed during the second author's 2006-2007 sabbatical at MIT.}}
\date{}

\begin{document}

\maketitle

\vspace*{-.3in}
\begin{center}
{\em Dedicated to Dennis Stanton
on the occasion of his 60th birthday}
\end{center}

\begin{abstract}
We determine how the flag $f$-vector
of any graded poset changes under the Rees product with
the chain, and more generally, any
$t$-ary tree.  As a corollary, the M\"obius function
of the Rees product of any graded poset with the chain,
and more generally, the $t$-ary tree,
is exactly the same as the Rees product of its dual with the chain,
respectively, $t$-ary chain.
We then
study enumerative and homological properties of the
Rees product of the cubical lattice with the
chain.
We give a bijective proof that the M\"obius function 
of this poset can
be expressed as $n$ times 
a signed derangement number.
From this we derive a new bijective proof of Jonsson's result that the
M\"obius function of the Rees product of the Boolean algebra
with the chain is given by a derangement number.
Using poset homology techniques
we find an explicit basis for the reduced
homology
and determine a representation
for the reduced homology of the order complex
of the Rees product of the cubical lattice
with the chain
over the symmetric group.

\vspace{.1in}
\end{abstract}


\section{Introduction}
\label{section_introduction}
\setcounter{equation}{0}

Bj\"orner and Welker~\cite{Bjorner_Welker}
initiated a study to generalize
concepts from commutative algebra to
the area of poset 
topology.
Motivated by the ring-theoretic Rees algebra,
one of the new poset operations they define is the Rees product.

\begin{definition}
For two graded posets
$P$ and $Q$ 
with rank function $\rho$
the
{\em Rees product},
denoted
$P * Q$,
is the set of ordered pairs $(p,q)$ in the Cartesian
product $P \times Q$ with $\rho(p) \geq \rho(q)$.  
These pairs are
partially ordered by $(p,q) \leq (p',q')$ if 
$p \leq_{P} p'$, 
$q \leq_{Q} q'$,
and 
$\rho(p') - \rho(p) \geq \rho(q') - \rho(q)$.
\end{definition}
The rank of the resulting poset is 
$\rho(P * Q) = \rho(P)$.
For more details concerning the Rees product and
other poset products, see~\cite{Bjorner_Welker}.

From the perspective of
topological combinatorics,
one of the most important
results that 
Bj\"orner and Welker
show in their paper
is that
the poset theoretic Rees product 
preserves the 
Cohen-Macaulay property; see~\cite{Bjorner_Welker}.

\begin{theorem}
[Bj\"orner-Welker]
\label{theorem_Bjorner_Welker}
If $P$ and $Q$ are two Cohen-Macaulay posets then
so is the Rees product $P*Q$.
\end{theorem}

Very little is known about the Rees product of specific 
examples of Cohen-Macaulay
posets.  However, what has been studied 
has yielded rich
combinatorial results.
The first example in this vein
is due to
Jonsson~\cite{Jonsson},
who settled an open question of Bj\"orner and Welker
concerning the Rees product of the Boolean algebra with the chain.
For brevity, throughout we will use the notation
$\Rees(P,Q)$ to denote the Rees product
$$
   \Rees(P,Q) = ((P - \{\hz\}) * Q) \cup \{\hz, \ho\}.
$$
As usual, we will assume that $P$ and $Q$ are graded posets with
$P$ having unique 
minimal element $\hz$
and unique maximal element $\ho$.

\begin{theorem}
[Jonsson]
\label{theorem_Jonsson}
The M\"obius function of the Rees product of the
Boolean algebra $B_n$ on $n$ elements with the $n$ element chain
$C_n$
is given by 
the $n$th derangement number, that is,
$$
  \mu(\Rees(B_n,C_n)) =  (-1)^{n+1} \cdot D_n.
$$
\end{theorem}
Recall the $n$th derangement number $D_n$ is the number of
permutations in the symmetric group $\Sym_n$ on $n$ elements having no
fixed points.
Classically
$D_n = \lfloor \frac{n!}{e} \rceil$ for $n \geq 1$ where
$\lfloor \cdot \rceil$ denotes the nearest integer function.
Jonsson's original proof uses
an non-acyclic element matching
to show the Euler characteristic vanishes appropriately.

The paper is organized as follows.
In the next section we begin by expressing the
flag $f$-vector of the Rees product of any graded
poset with a $t$-ary tree in terms of the flag $f$-vector of
the original poset.  We obtain the surprising conclusion that the
M\"obius function of the poset with the tree coincides
with the M\"obius function of its dual with the tree.
We then
study the signed version of Jonsson's results,
that is,
the Rees product of the 
rank $n+1$ cubical lattice $\mathscr{C}_n$,
(i.e., the face lattice of the $n$-dimensional cube)
with the $n+1$ element chain $C_{n+1}$.
Using poset techniques, we give explicit formulas for the
for the M\"obius function
of $\Rees({\mathscr C}_n, C_{n+1})$
and show its
M\"obius function equals $(-1)^n \cdot n \cdot D_{n-1}^{\pm}$.
Here 
$D_n^{\pm}$ is the signed derangement number with
$D_n^{\pm} = \lceil \frac{2^{n-1} (n-1)!}{\sqrt{e}}\rfloor$
for $n \geq 1$.
As a corollary to our enumerative results, we give an
explicit bijective 
proof of Jonsson's theorem.
We then find an explicit basis for the reduced
homology of the order complex 
of $\Rees(\mathscr{C}_n,C_{n+1})$
and determine a representation
of the reduced homology of
this order complex 
over the symmetric group. 
In the last section we end with further questions.

\section{Rees product of graded posets with a tree}
\setcounter{equation}{0}

In this section we determine the flag vector of 
the Rees product of any graded
poset~$P$ with a $t$-ary tree.
As a consequence we show the M\"obius function of the Rees products
$\Rees(P,T_{t,n+1})$ and $\Rees(P^*,T_{t,n+1})$ coincide,
although the posets are not isomorphic in general.

For nonnegative integers $n$ and $t$, let
$T_{t,n+1}$ be the poset corresponding to
a  $t$-ary tree of rank $n$,
that is,
the poset consisting of $t^k$ elements of rank $k$
for $0 \leq k \leq n$ with each nonleaf element
covered by exactly $t$ children.
Observe that the $1$-ary tree
$T_{1,n+1}$ is precisely the $(n+1)$-chain $C_{n+1}$.
Recall for a graded poset $P$ of rank $n+1$
and $S = \{s_1, \ldots, s_k\} \subseteq \{1, \ldots, n\}$
with $s_1 < \cdots < s_k$,
the {\em flag $f$-vector}
$f_S = f_S(P)$ is the
number of chains
$\hz < x_1 < \cdots < x_k < \ho$
with $\rho(x_i) = s_i$.

We now define two weight functions.
Here we use the notation
$[k]$ to denote the $t$-analogue of the
nonnegative integer $k$, i.e.,
$[k] = 1 + t + \cdots + t^{k-1}$.

\begin{definition}
For a nonempty subset $S = \{s_1 < \cdots < s_k\} \subseteq \Ppp$
define 
$$ 
    \wwt(S) = [s_1] \cdot [s_2 - s_1 +1] \cdots [s_k - s_{k-1} +1]
$$
with $\wwt(\emptyset) = 1$.
For a nonempty subset
$S = \{s_1 < \cdots < s_k\} \subseteq \{1, \ldots, n\}$
define 
$$ 
    \vwt(S) = \wwt(S \cup \{n+1\}) - \wwt(S)
            = t \cdot \wwt(S) \cdot [(n+1) - s_k] 
$$
with $\vwt(\emptyset) = t \cdot [n]$.
\end{definition}

\begin{lemma}
\label{lemma_flag_weights}
For a graded poset $P$ of rank $n+1$, let
$R = \Rees(P,T_{t,n+1})$.
Then the flag $f$-vector of the poset
$R$ is given by
\begin{eqnarray}
    f_S(R) &=& \wwt(S) \cdot f_S(P),  \label{equation_usual_ranks}\\
    f_{S \cup \{n+1\}}(R) &=& \wwt(S \cup \{n+1\}) \cdot f_S(P),
                            \label{equation_extra_rank}
\end{eqnarray}
for $S \subseteq \{1, \ldots, n\}$.
\end{lemma}
\begin{proof}
Consider first $S = \{s_1 < \cdots < s_k\} \subseteq \{1, \ldots, n\}$.
Given an element $x_1$ of rank $\rho(x_1) = s_1$
from the poset $P$,
there are $[s_1]$ copies of it
in the Rees poset~$R$.
Each of these copies has $[s_2 - s_1 + 1]$
elements in $R$ of rank $s_2$ which are greater than it
with respect to the partial order of the Rees poset~$R$.
In general, each rank $s_i$ element
in $R$ has
$[s_{i+1} - s_i +1]$ elements greater than it in the
Rees poset~$R$.
Hence relation~(\ref{equation_usual_ranks}) holds.

To show~(\ref{equation_extra_rank}), note the maximal
element $\ho$ of $P$ gets mapped to the $[n+1]$ coatoms
of the Rees poset $R$. 
In particular $[(n+1) - s_{k} + 1]$ of these elements
will cover a given element  of rank $s_k$ in~$R$.
Hence the result follows.
\end{proof}

\begin{lemma}
\label{lemma_mu_Rees_poset}
For a graded poset $P$ of rank $n+1$, let
$R = \Rees(P,T_{t,n+1})$.
Then
$$
     \mu(R) = \sum_{S \subseteq \{1, \ldots, n\}} 
                    (-1)^{|S|} \cdot \vwt(S) \cdot f_S(P).
$$
\end{lemma}
\begin{proof}
By Philip Hall's theorem, we have
\begin{eqnarray*}
     \mu(R) &=& \sum_{S \subseteq \{1, \ldots, n+1\}} (-1)^{|S|-1} f_S(R)\\
            &=& \sum_{S \subseteq \{1, \ldots, n\}} (-1)^{|S|-1} f_S(R)
             +  \sum_{S \subseteq \{1, \ldots, n\} }
                     (-1)^{|S|} f_{S \cup \{n+1\}}(R)\\
            &=& \sum_{S \subseteq \{1, \ldots, n\} }
                     (-1)^{|S|-1} \wwt(S) \cdot f_S(P)
             +  \sum_{S \subseteq \{1, \ldots, n\} }
                     (-1)^{|S|} \wwt(S \cup \{n+1\}) \cdot f_S(P),
\end{eqnarray*}
where we have expanded the flag $f$-vector of 
the poset $R$ 
using Lemma~\ref{lemma_flag_weights}.
Combining the two sums proves the desired identity.
\end{proof}

\begin{theorem}
\label{theorem_poset_and_dual_with_tree}
For a graded poset $P$ of rank $n+1$ we have
$$
   \mu(\Rees(P,T_{t,n+1})) =    \mu(\Rees(P^*,T_{t,n+1})), 
$$
where $P^*$ is the dual of $P$.
In particular, for the chain on $n+1$ elements we have
$$
   \mu(\Rees(P,C_{n+1})) =    \mu(\Rees(P^*,C_{n+1})). 
$$
\end{theorem}
\begin{proof}
Let $S = \{s_1 < \cdots < s_k\} \subseteq \{1, \ldots, n\}$.
The result follows by noting
that
$\vwt(S) = \vwt(S^{\rm rev})$,
where the reverse of $S$ is
$S^{\rm rev} = \{n + 1 - s_k, n+1 - s_{k-1},  \ldots,  n+1 - s_1\}$
and applying
Lemma~\ref{lemma_mu_Rees_poset}.
\end{proof}

It is clear from the definition of the weight
$\vwt(S)$ that the M\"obius function $\mu(\Rees(P,T_{t,n+1}))$ 
is divisible by $t$. When the poset has odd rank
we can say more.
\begin{corollary}
For a graded poset $P$ of odd rank $n+1$,
the M\"obius function $\mu(\Rees(P,T_{t,n+1}))$ 
is divisible by
$[2] = 1 + t$.  
In particular,
for a graded poset $P$ of odd rank $n+1$,
the M\"obius function $\mu(\Rees(P,C_{n+1}))$ 
is even.
\end{corollary}
\begin{proof}
Observe that $1+t$ divides $[k]$ if and only if
$k$ is even.
Hence $1+t$ does not divide
$\vwt(S)$ for a
set $S = \{s_{1} < \cdots < s_{k}\}$
implies that
$s_{1}$ is odd,
$s_{i}$ has the same parity as $s_{i+1}$
and $n+1-s_{k}$ is odd.
This implies that $n$ is odd.
Hence that $n$ is even implies that
the weight $\vwt(S)$ is divisible by $1 + t$  for all
subsets $S$, including the empty set.
Thus by
Lemma~\ref{lemma_mu_Rees_poset}
the M\"obius function of
$\Rees(P,T_{t,n+1})$ is divisible by $1 + t$.
\end{proof}

\section{Rees product of the cubical lattice with the chain}
\setcounter{equation}{0}

In this section we give an explicit formula for the M\"obius function
of the poset $\Rees({\mathscr C}_n, C_{n+1})$.
After finding an $R$-labeling
in Section~\ref{section_R_labeling}, we relate the M\"obius function
with a class of permutations, that is,
the double augmented barred signed permutations.
These are in a one-to-one correspondence with certain
skew diagrams.
We will return to these 
when we consider homological questions
for $\Rees({\mathscr C}_n, C_{n+1})$.
In Section~\ref{section_bijection} we give a bijective
proof of the M\"obius function result expressed
as a permanent of a certain matrix.

We represent an element 
$(x,i) \in \Rees({\mathscr C}_n, C_{n+1}) - \{\hz, \ho\}$ 
as an ordered pair
where
the $n$-tuple
$x = (x_1,x_2,\ldots,x_n)\in \{0,1,*\}^n$
and
$i \in \{1, \ldots, n\}$.  
Observe that such an element $(x,i)$
has rank $k$ if there are
exactly $k-1$ stars appearing in its first coordinate,
$1 \leq i \leq k$.

For a graded poset $P$ with minimal element $\hz$ and
maximal element $\ho$, throughout we will use
the shorthand $\mu(P)$ to denote
the M\"obius function
$\mu_P([\hz,\ho])$.

Proposition~\ref{proposition_recursion}
gives an explicit formula for the M\"obius function
of the poset $\Rees({\mathscr C}_n, C_{n+1})$.
The proof will require
the following lemma.

\begin{lemma}
\label{lemma_convolution}
The following 
identity holds:
$$
     1 + \sum_{k=0}^{n} {n \choose k} (-1)^{k+1} k! (n-k+1) = 0.
$$
\end{lemma}
\begin{proof}
Define sequences $(a_n)_{n \geq 0}$
and
$(b_n)_{n \geq 0}$ by 
$a_n = (-1)^{n+1} n!$ and 
$b_n = n+1$.
These sequences have exponential generating functions
\begin{eqnarray*}
     A(x) = \sum_{n \geq 0} (-1)^{n+1} x^n = - \frac{1}{1+x}
\end{eqnarray*}
and
\begin{eqnarray*}
     B(x) = \sum_{n \geq 0} (1+n) \frac{x^n}{n!}= (1+x)e^x.
\end{eqnarray*}
Thus, $D(x)=A(x)B(x)=-e^x$.  But
\begin{eqnarray*}
     D(x) &=& \sum_{n\geq 0} 
              \sum_{k=0}^{n} {n \choose k}a_k b_{n-k} \frac{x^n}{n!}\\
          &=& \sum_{n\geq 0} 
              \sum_{k=0}^{n} {n \choose k} (-1)^{k+1} k! (n-k+1)
              \frac{x^n}{n!},
\end{eqnarray*}
which proves the claim.
\end{proof}

\begin{proposition}
\label{proposition_recursion}
The M\"obius function of the Rees product of the cubical
lattice with the chain is given by
$$
     \mu(\Rees({\mathscr C}_n, C_{n+1})) = -1 + \sum_{i=0}^{n}
     (-1)^{n-i} \cdot 2^{n-i} {n \choose i} (i+1)(n-i)!.
$$
\end{proposition}
\begin{proof}
Let $x$ be an element of corank $k$ from $\Rees({\mathscr C}_n, C_{n+1}) - \{\hz, \ho\}$.
First note 
that the number of elements of corank $i$  
in the half-open interval
$[x,\ho)$
is 
${k-1 \choose i-1} \cdot (k-i+1)$.
This follows from the fact that
the element
$x = (b,p)$ has $k-1$ non-stars
appearing in $b$,
so a corank $i$ element
$y = (c,q) \in [x,\ho)$
has
$i-1$ more stars appearing in $c$ and
the second coordinate
$q$ satisfying
$p \leq q \leq p+k-i+1$.
Hence there are  
${k-1 \choose i-1} \cdot (k-i+1)$
such elements $y$.
Secondly, we claim that
for a corank $k$ element $x \in \Rees({\mathscr C}_n, C_{n+1}) - \{\hz, \ho\}$,
we have
\begin{equation}
\label{equation_Mobius_x}
     \mu([x,\ho]) = (-1)^k \cdot (k-1)!.
\end{equation}
We induct on the corank $k$.
The case $k=0$ is clear, as then $x$ is a coatom.
For the general case,
we have
\begin{eqnarray*}
     \mu([x,\ho]) & = & - \sum_{x < y \leq \ho} 
                                   \mu([y,\ho])\\
                      & = &  - \left(1 + \sum_{
\onethingatopanother{x < y \leq \ho,}{1 \leq \corank(y) \leq k-1}
                                               }
                                    \mu([y,\ho])
                               \right)\\
                      & = &  - \left(1 + \sum_{i =1}^{k-1}
                                    (-1)^i \cdot (i-1)! \cdot 
                                    \mbox{number of elements of corank $i$
                                         in $[x,\ho)$}
                               \right),
\end{eqnarray*}
where the third equality is applying the induction hypothesis.
The number of corank $i$ elements in the
half-open interval
$[x,\ho)$ is
${k-1 \choose i-1} \cdot (k-i+1)$, 
giving
\begin{eqnarray*}
     \mu([x,\ho]) & = &  - \left(1 + \sum_{i =1}^{k-1}
                                    (-1)^i {k-1 \choose i-1}
                                    \cdot (i-1)! 
                                    \cdot (k-i+1)
                               \right)
                        = (-1)^k \cdot (k-1)!
\end{eqnarray*}
by
Lemma~\ref{lemma_convolution}.

To finish the argument, there are
$2^{n-k} \cdot {n \choose k} \cdot (k+1)$
elements of rank $k+1$, each having M\"obius
value
$\mu(x,\ho) = (-1)^{n-k+1} \cdot (n-k)!$.
Hence the lemma follows
the fact that for a poset $P$ with $\hz$ and~$\ho$, the
identity 
$\mu_{P}(\hz,\ho) = - \sum_{\hz < x \leq \ho} \mu_{P} (x,\ho)$
holds.
\end{proof}

\begin{table}
\begin{center}
\begin{tabular}{crrrl}
$n$ &
$D_n = (-1)^{n+1} \mu(\Rees(B_n, C_n)$
&$(-1)^n \mu(\Rees({\mathscr C}_n, C_{n+1}))$ &$=$ &Factorization\\
\hline
0 & 1 & 0&=& 0\\
1 & 0 & 1 &=& $1\cdot 1$\\
2 & 1 & 2 &=& $2\cdot 1$\\
3 & 2 & 15&=&$3\cdot 5$\\
4 & 9 & 116&=&$4\cdot 29$\\
5 & 44 & 1165&=&$5 \cdot 233$\\
6 & 265 & 13974&=& $6 \cdot 2329$\\
7 & 1854 & 195643&=& $7 \cdot 27949$\\
8 & 14833 & 3130280&=& $8\cdot 391285$\\
9 & 133496 & 56345049&=& $9 \cdot 6260561$\\
10 & 1334961 & 1126900970 &=&$10 \cdot 112690097$\\
\end{tabular}
\end{center}
\caption{Table of M\"obius values for the
Rees product of the Boolean algebra
with the chain and
the Rees product of the
cubical lattice
with the chain.}
\label{mobius_rees_boolean_cubical}
\end{table}

\tinyspace
\section{Edge labeling}
\label{section_R_labeling}
\setcounter{equation}{0}

We begin by recalling some facts
about $R$-labelings.  For a complete overview,
we refer the reader to Section 5 of Bj\"orner and Wachs'
paper~\cite{Bjorner_Wachs}.

Given  a poset $P$ an {\em edge labeling}
is a map $\lambda: E(P) \rightarrow \Lambda$,
where $E(P)$ denotes the edges in the Hasse diagram of
$P$ and the labels are elements from a poset $\Lambda$.
An edge labeling $\lambda$ 
is said to be an {\em $R$-labeling}
if
in every interval $[x,y]$ of $P$
there is a unique saturated chain 
$c: x=x_0 \prec x_1 \prec \cdots \prec x_k = y$
whose labels
are rising, that is,
which satisfies
$\lambda(x_0,x_1) <_{\Lambda}
 \lambda(x_1,x_2) <_{\Lambda}
\cdots
<_{\Lambda}
 \lambda(x_{k-1},x_k).$ 
Given a maximal chain 
$m : \hz = x_0 \prec x_1 \prec \cdots \prec x_n = \ho$ 
in $P$, the 
{\em descent set} of $m$ is the 
set
$D(m) = \{i  :  \lambda(x_{i-1},x_i) \not<_{\Lambda} 
                         \lambda(x_{i},x_{i+1}) \}$.
Alternatively,
when we view the labels
of the maximal chain
as the word
$\lambda(m) = \lambda_1 \cdots \lambda_n$,
where
$\lambda_i = \lambda(x_{i-1},x_i)$
and the rank of $P$ is $n$,
there is a descent 
in the $i$th position of $\lambda(m)$ if
the labels $\lambda_i$ and $\lambda_{i+1}$ are either
incomparable in the label poset $\Lambda$ or
satisfy
$\lambda_i >_{\Lambda} \lambda_{i+1}$.
In particular, a maximal chain $m$ is said to be
{\em rising} 
if its descent set satisfies $D(m) = \emptyset$
and
{\em falling} if 
$D(m) = \{1, \ldots, n\}$.

The usefulness of an $R$-labeling is that it gives
an alternate way to compute the M\"obius function $\mu$ of a poset.
Variations of this result are due to Stanley in the case of 
admissible lattices,
Bj\"orner for $R$-labelings and edge lexicographic labelings,
and Bj\"orner--Wachs for non-pure posets with a
$CR$-labeling.
See~\cite{Bjorner_Wachs} for historical details.

\begin{theorem}
\label{theorem_mobius_chains}
Let $P$ be a graded poset of rank $n$ with an $R$-labeling.
Then with respect to this
$R$-labeling the M\"obius function is given by
$$
    \mu(\hz,\ho) = (-1)^n \cdot \mbox{number of falling maximal chains in $P$}.
$$
\end{theorem}

Let 
$\lambda : E(\Rees({\mathscr C}_n, C_{n+1})) \rightarrow 
 \{0, \pm 1, \pm 2, \ldots, \pm n, n+1\}  \times \{0,1\}$ 
be a labeling of the edges of the Hasse
diagram of $\Rees({\mathscr C}_n, C_{n+1})$ defined by
$$ 
     \begin{array}{lclclc}
     & \mbox{Edge}  & &\mbox{Condition} &\lambda(E) & \mbox{Notation}
     \\
     \hline
     (x,i) & \prec & (y,i) &x_a =1, y_a = * & (a,0) & a
     \\
     (x,i) & \prec & (y,i) &x_a = 0, y_a = * &  (-a,0) & -a
     \\
     (x,i) & \prec & (y,i+1) &x_a =1,  y_a = * & (a,1) & \overline{a}
     \\
     (x,i) & \prec & (y,i+1) &x_a = 0, y_a = * & (-a,1) & \overline{-a}
     \\
     \hz & \prec & (x,1) &                    &(0,0) & 0
     \\
     (x,i) & \prec & \ho &                    &(n+1,0) & n+1
\end{array} 
$$
where 
$x=(x_1, \ldots, x_n)$
and
$y=(y_1, \ldots, y_n)$.
The elements 
$\displaystyle{\{0, \pm 1, \ldots, \pm n, n+1\} \times \{0,1\}}$ 
are partially ordered with the
  product order, that is 
$(x,i)\leq (y,j)$ if 
$x \leq y$ and $i \leq j$.

\begin{proposition}
The labeling $\lambda$ is an $R$-labeling of $\Rees({\mathscr C}_n, C_{n+1})$.
\end{proposition}

\begin{proof}
Let 
$I=[(x,i),(y,j)]$
be an interval in $\Rees({\mathscr C}_n, C_{n+1}) - \{\hz,\ho\}$
of length $m$
with
$x = (x_1, \ldots, x_n)$
and
$y = (y_1 \ldots, y_n)$.
We wish to find a
saturated chain 
$c: (x,i)=(z_0,p_0) \prec (z_1,p_1) \prec \cdots \prec (z_m,p_m)=(y,j)$ 
in the interval $I$
with increasing edge labels.

Let 
$S_0 = \{k : x_k=0 \mbox{ and } y_k=*\}$ 
and
$S_1 = \{k : x_k=1 \mbox{ and } y_k=*\}$.
Let $s = j-i$ and $t = |S_0|$.
Without loss of generality, we may
assume
$S_0 = \{i_1, \ldots, i_t\}$
and
$S_1 = \{i_{t+1}, \ldots, i_m\}$
with
$i_1 > \cdots > i_t$
and
$i_{t+1} < \cdots < i_m$.
Set
$(z_0, p_0) = (x, i)$.
For $1 \leq k \leq m$,
let $(z_k,p_k)=((z_{1,k}, \ldots, z_{n,k}),p_k)$ 
where
$$ 
    z_{i,k} = \left\{\begin{array}{ll} 
                          *         & \mbox{ if $i = i_k$}, \\
                          z_{i,k-1} & \hbox{ otherwise}, 
                    \end{array} \right. 
$$
and
$$ 
     p_k = \left\{ \begin{array}{ll}
                 p_{k-1}     &  \hbox{ if $1 \leq k \leq m-s$},\\
                 p_{k-1} + 1 & \hbox{ otherwise.} 
\end{array} \right. 
$$
The first coordinate of the edge labels of the chain $c$
form
the strictly increasing sequence
$-i_1 < \cdots < -i_t < i_{t+1} < \cdots < i_m$
as the $i_j$'s are all positive,
while the second coordinate of the edge labels
form the weakly increasing sequence
$0 \leq \cdots \leq 0 \leq 1 \leq \cdots \leq 1$.
Hence the chain $c$ constructed is increasing.

We also claim that the chain $c$ 
is the unique such chain that is increasing
in the interval $I$.
For any maximal chain in this interval,
each $i \in S_0$ appears as the first
coordinate in an edge label with a negative sign and 
every $i \in S_1$
must appear with a positive sign.
Hence
there is exactly one way to linearly order these $m$ values.
The second coordinate of the
labels of any maximal chain in $I$
is a permutation of the multiset
$\{0^{m-s}, 1^s\}$.
Again, there is exactly one way to order these $m$ values
in a weakly increasing fashion.
Hence the increasing chain $c$ is unique.

For the case when the interval is $[\hz, (y,j)] \in \Rees({\mathscr C}_n, C_{n+1})$ with
$(y,j) \neq \ho$,
the first edge label in any saturated chain is always
$(0,0)$.
Hence the first coordinate of the
labels in any increasing chain in
this interval must all be non-negative, implying an
increasing chain must pass through the
atom $(a,1) = ((1, \ldots, 1), 1)$.
The remainder of the increasing chain
is given by the unique increasing
maximal chain in the interval
$[(a,1), (y,j)]$.

For an interval of the form $[(x,i),\ho]$, 
since the last edge label of any saturated chain
has label
$(n+1, 0)$,
this forces all the elements of such a chain
to be of the form $(y,i)$ with
$x \leq_{\mathscr{C}_n} y$.
In particular, the rank~$n$ element of such a chain is
precisely the element
$(b,i) = ((*, \ldots, *), i)$.
Hence the increasing maximal chain in
$[(x,i),\ho]$ is given by the increasing
maximal chain guaranteed in
$[(x,i),(b,i)]$ concatenated with the element $\ho$.
\end{proof}

\tinyspace
\section{Falling chains}
\setcounter{equation}{0}

Define the set of ({\em double augmented}) {\em barred signed  permutations}
$\barredsigned{n}$
to be
those permutations
$\pi=\pi_0 \pi_1 \cdots \pi_{n+1}$ satisfying
($i$) $\pi_0 = 0$ and
$\pi_{n+1} = n+1$,
($ii$) for $1 \leq i \leq n$,
$\pi_i$ is equal to
one of 
$a_i$, $-a_i$,
$\overline{a_i}$ or
$\overline{-a_i}$ 
for some $a_i \in \{1, \ldots, n\}$,
and
($iii$) $a_1 \cdots a_n$ is a permutation
in the symmetric group~$\Sym_n$
on $n$ elements. 
Given a double augmented barred signed permutation
$\pi=\pi_0 \pi_1 \cdots \pi_{n+1}$,
a {\em descent} at position $i$ occurs
when 
$|\pi_i| > |\pi_{i+1}|$,
where $|\pi_j|$ denotes the element $\pi_j$ with its (possible) bar removed
and sign preserved.

\begin{proposition}
\label{proposition_chains_and_permutations} 
With respect to the $R$-labeling $\lambda$
of the poset $\Rees({\mathscr C}_n, C_{n+1})$,
the falling chains are described
as 
the set of double augmented barred signed permutations
$\pi =\pi_0 \pi_1 \cdots \pi_{n+1} \in \barredsigned{n}$
satisfying
\begin{enumerate}

\item 
if $\pi_i$ is unbarred 
then there must be a descent at the 
$i$th position.

\item if $\pi_i$ is barred, 
then either
($i$) $\pi_{i+1}$ is unbarred or
($ii$) $\pi_{i+1}$ is barred and there is a descent at the $i$th position.
\end{enumerate}
\end{proposition}

\begin{example}
The permutation 
$(0,-3,\overline{-4},2,\overline{-1},5)\in \barredsigned{4}$
corresponds to the falling chain 
$$\hz \prec
(0100,1) \prec 
(01*0,1) \prec
(01**,2) \prec
(0***,2) \prec
(****,3) \prec
\ho$$
in the poset $\Rees({\mathscr C}_4, C_5)$.
\end{example}

\begin{proof_special}{Proposition~\ref{proposition_chains_and_permutations}}
Given a barred signed permutation 
satisfying the conditions
of the proposition,
we wish to find a falling
chain 
$c: \hz \prec (x_1,i_1) \prec  \cdots \prec (x_n,i_n) \prec \ho$ 
in $\Rees({\mathscr C}_n, C_{n+1})$.  
For $1 \leq k \leq n$, if $\pi_k < 0$
then set $x_{1,k} = 1$;
otherwise set
$x_{1,k} = 0$.
To find $(x_k,i_k)$
recursively, 
set $i_1 = 0$,
let
$x_{w_k,k}= *$, and set 
$$
     i_k= \left\{\begin{array}{ll}
               i_{k-1}+1 & \mbox{if } \pi_k \mbox{ is barred,}\\
               i_{k-1} & \mbox{if } \pi_k \mbox{ is not barred.}
          \end{array}
          \right.
$$
Observe that  $c$ is a falling chain.  
The labels on the barred signed
permutation correspond to the labels on the  falling chain.
Note that if the unbarred signed permutation 
does not have a descent at some position~$k$,
then $\pi_k$ is barred and $\pi_{k+1}$ is not, implying the second
coordinate in the labeling 
$\lambda ((x_k,i_k), (x_{k+1},i_{k+1}))$ is $1$,
while the second coordinate in the labeling
$\lambda ((x_{k+1},i_{k+1}), (x_{k+2},i_{k+2}))$ is $0$.
Hence,
the chain is not rising in the $k$th position.  Otherwise, the
unbarred permutation has a descent and hence the first coordinate
in the labeling $\lambda ((x_k,i_k), (x_{k+1},i_{k+1}))$
is greater than the first coordinate in the labeling
$\lambda ((x_{k+1},i_{k+1}), (x_{k+2},i_{k+2}))$ and hence the
chain is not rising.
\end{proof_special}

Throughout we will use
$\Fall_n$ to denote the set of all the
falling double augmented barred signed permutations
in $\barredsigned{n}$.

\begin{theorem}
The M\"obius function of the Rees product $\Rees({\mathscr C}_n, C_{n+1})$
is given by
$$
   \mu(\Rees({\mathscr C}_n, C_{n+1}))= (-1)^n \cdot {\sum_{c} 2^{n-c_1} 
                   {n \choose {c_1, \ldots, c_k}} \cdot 
                   c_1 \cdot \prod_{i=2}^{k}(c_i -1)},
$$
where the sum is over all
compositions $c = (c_1, \ldots, c_k)$ 
of $n$ and 
$1 \leq k \leq n$.
\end{theorem}

\begin{proof}
By Theorem~\ref{theorem_mobius_chains}, 
to determine the M\"obius function
of the poset $\Rees({\mathscr C}_n, C_{n+1})$
it is enough to count the number
of falling chains in $\Rees({\mathscr C}_n, C_{n+1})$.  
Proposition~\ref{proposition_chains_and_permutations}
allows one to separate 
the double augmented barred signed permutations
corresponding to falling chains 
into substrings which consist of a sequence of unbarred
elements followed by a sequence of barred elements.

By Proposition~\ref{proposition_chains_and_permutations},
the element
$0$ will alway be part of the 
first substring and
the last substring will consist only of the element
$n+1$.
Determining the size of
each substring
is equivalent to taking a
composition $c=(c_1,c_2,\ldots,c_k)$ of $n$. 
Note that the first
substring will be of size $c_1 + 1$ to account for 
the element $0$ and the
$(k+1)$st substring will consist only of the element  $n+1$.

In each substring
there is a sequence of elements without bars followed by
a sequence of elements with bars.  Given the size of each
substring we determine at what place the barred
elements begin.  In the
first substring we can begin the bars at any place, so 
there are $c_1$ ways.  For
all the other substrings the first element cannot be barred, 
for otherwise  it would
belong to the previous substring.
Thus, we can begin the sequence of barred elements 
in $c_i -1$
ways
for $i = 2, \ldots, k$.  
The total number of ways to place bars over elements is
${\displaystyle c_1 \cdot \Pi_{i=2}^{k}(c_i -1)}$.

Next, we choose the elements that will be in each 
substring.  This is done
in 
${n\choose {c_1,c_2,\ldots,c_k}}$ ways.
Now we must sign these elements.  Note
that the elements in each substring must
be arranged in decreasing order.  Once we have chosen the signs,
this can be done in exactly one way.
Furthermore, all of the elements in the
first block must be negative because 
the falling double augmented signed permutation
begins with the element $0$.
This leaves
$2^{n-c_1}$ ways to sign the remaining elements.
\end{proof}

\section{Signed derangement numbers, skew diagrams and a bijective proof}
\label{section_bijection}
\setcounter{equation}{0}

Recall that the derangement number
$D_n$ can be expressed as the permanent
of an $n \times n$ matrix having~$0$'s on the
diagonal and $1$'s everywhere else.
Motivated by this, define the
{\em signed derangement number}~$D_n^{\pm}$ 
by
$$
D_n^{\pm} =
\hbox{{\rm per}} \begin{bmatrix}
1      & 2      & \cdots  & 2\\
2      & 1      & \cdots  & 2\\
\vdots &\vdots  &\ddots   & \vdots\\
2      & 2      &\cdots   & 1
\end{bmatrix},
$$
that is,
the permanent of an
$n \times n$ matrix having
$1$'s on the diagonal and
$2$'s everywhere else.
It is straightforward to see that this permanent enumerates
signed permutations 
$\pi = \pi_1 \cdots \pi_n \in \signed{n}$ 
having no fixed points, that is,
no index
$i$ satisfying $\pi_i = i$.
See~\cite{Chen_Zhang,Chow} for 
details.

\begin{lemma}
For $n \geq 0$,
$D_n^{\pm}$ is the nearest integer to
$\frac{2^n \cdot n!}{\sqrt{e}}$.
\end{lemma}
\begin{proof}
This follows directly from the generating function
$\sum_{n \geq 0} D_n^{\pm} \frac{x^n}{n!} = \frac{e^{-x}}{1-2x}$.
\end{proof}

In this section we give a bijective
proof of  the following theorem.

\begin{theorem}
\label{theorem_permanent}
The M\"obius function of the Rees product of the cubical
lattice with the chain is given by
$$
 \mu(\Rees({\mathscr C}_n, C_{n+1})) =
(-1)^n \cdot  n \cdot D_{n-1}^{\pm}.
$$
\end{theorem}
As a corollary
to Theorem~\ref{theorem_permanent}, 
we can slightly modify our proofs to give a bijective
proof of Jonsson's result (Theorem~\ref{theorem_Jonsson}).
\begin{corollary}
There is an explicit bijection 
implying that 
$$
  \mu(\Rees(B_n,C_n)) = (-1)^{n+1} \cdot D_n.
$$
\end{corollary}

In order to prove Theorem~\ref{theorem_permanent},
we will work with skew diagrams
associated to falling double augmented barred and signed permutations.
In Section~\ref{section_basis} we will use these skew diagrams to describe 
$\Delta(\Rees({\mathscr C}_n, C_{n+1}))$,
the order
complex of the Rees product of the cubical lattice with the chain,
in the spirit of Wachs' work with the 
$d$-divisible partition lattice~\cite{Wachs}.
We will also use these diagrams to construct an explicit basis
for the homology of $\Rees({\mathscr C}_n, C_{n+1})$.

Besides the interest in the bijection itself
to prove Theorem~\ref{theorem_permanent}, these diagrams
allow
us  to find explicit bases for
the integer homology
$\widetilde{H}_n(\Delta(\Rees(\mathscr{C}_n,C_{n+1})),\Zzz)$ 
indexed by the falling augmented signed barred
permutations.

We begin by recalling some 
objects from combinatorial representation theory.
For background material in this area,
we refer to Sagan's book~\cite{Sagan}.
Let $(\lambda_1, \ldots, \lambda_k) \vdash n$ be a partition of the
integer~$n$ with $\lambda_1 \leq \cdots \leq \lambda_k$.
Recall the Ferrers diagram of $\lambda$ consists of~$n$ 
boxes where row $i$ has $\lambda_i$ boxes 
for $i = 1, \ldots, k$ and all the
rows are left-justified.
Given two Ferrers diagrams
$\mu \subseteq \lambda$, the
{\em skew diagram} $\lambda / \mu$
is the set of all boxes
$\lambda/\mu = \{b \:\: : \:\: b \in \lambda \mbox{ and } b \notin \mu\}$.

\setcounter{figure}{0}
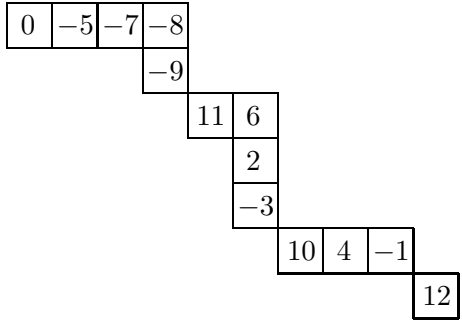
\begin{figure}
\label{figure_diagram}
\setlength{\unitlength}{0.6 mm}
\begin{center}
\begin{picture}(100,70)(0,0)
\put(0,70){\line(1,0){40}}
\put(0,60){\line(1,0){40}}
\put(30,50){\line(1,0){30}}
\put(40,40){\line(1,0){20}}
\put(50,30){\line(1,0){10}}
\put(50,20){\line(1,0){40}}
\put(60,10){\line(1,0){40}}
\put(90,0){\line(1,0){10}}

\put(0,60){\line(0,1){10}}
\put(10,60){\line(0,1){10}}
\put(20,60){\line(0,1){10}}
\put(30,50){\line(0,1){20}}
\put(40,40){\line(0,1){30}}
\put(50,20){\line(0,1){30}}
\put(60,10){\line(0,1){40}}
\put(70,10){\line(0,1){10}}
\put(80,10){\line(0,1){10}}
\put(90,0){\line(0,1){20}}
\put(100,0){\line(0,1){10}}

\put(3,3){\put(0,60){$0$}}
\put(1,3){\put(10,60){$-5$}}
\put(1,3){\put(20,60){$-7$}}
\put(1,3){\put(30,60){$-8$}}
\put(1,3){\put(30,50){$-9$}}
\put(2,3){\put(40,40){$11$}}
\put(3,3){\put(50,40){$6$}}
\put(3,3){\put(50,30){$2$}}
\put(1,3){\put(50,20){$-3$}}
\put(2,3){\put(60,10){$10$}}
\put(3,3){\put(70,10){$4$}}
\put(1,3){\put(80,10){$-1$}}
\put(2,3){\put(90,0){$12$}}

\end{picture}
\end{center}
\caption{The skew diagram corresponding to
the falling double augmented barred signed permutation
$\pi = 0 \,\, -5 \,\,
-7 \,\, \overline{-8} \,\, \overline{-9} \,\, 
11 \,\, \overline{6} \,\, \overline{2} \,\,
\overline{-3} \,\, 10 \,\, 4 \,\, \overline{-1} \,\, 12$
in $\barredsigned{11}$.}
\end{figure}

For us, a {\em hook} is a skew diagram of the form
$\lambda / \mu $ where
$\lambda = ((h+1)^v)$ and
$\mu = (h^{(v-1)})$.
We will be interested in skew diagrams consisting of a disjoint
union of hooks.  More precisely, 
let 
$c = (c_1, \ldots, c_k)$ be a composition of $n$
with
$c_i = u_i + b_i$, for $i = 1, \ldots, k$
where $u_1 \geq 0$, $u_i > 0$ for $i = 2, \ldots, k$,
and $b_i > 0$ for $i= 1, \ldots, k$.
Form the partitions
$\lambda = (\lambda_1, \ldots, \lambda_k)$
and
$\mu = (\mu_1, \ldots, \mu_k)$
where
$\lambda_i =  (u_1 + \cdots + u_i  + i)^{b_i}$
for
$1 \leq i \leq k$,
$\mu_i = ((u_1 + \cdots + u_i + i -1)^{b_i - 1}, u_1 + \cdots + u_i + i)$
for $1 \leq i \leq k-1$,
and
$\mu_k = (u_1 + \cdots + u_k + k -1)^{b_k - 1}$.
The skew diagram $\lambda / \mu$ is then a union of
$k$ hooks where the southeast corner of the last
box of the $i$th hook touches the northwest corner of the first
box of the $(i+1)$st hook.  
We call such a diagram an 
{\em unsigned barred
permutation 
skew diagram}.
We call a filling of the 
$n$ boxes with the elements $\{1, \ldots, n\}$
{\em standard} 
if the rows are decreasing when read from left to right
and the columns are decreasing
when read from top to bottom.
If we insert a box labelled $0$ in front of the first horizontal row
and add a box labelled $n+1$ as the new last hook,
then we call such a filled diagram a
{\em standard double augmented unsigned barred skew diagram}.
Given a double augmented unsigned barred permutation that is falling,
recall that it consists of strings of
unbarred and barred elements concatenated together.
Given such a falling permutation,
one forms the 
standard skew diagram by representing  the 
first string of unbarred elements 
as the first horizontal string of boxes in the first hook
concatenated with 
the same number of vertical boxes
as the number of  barred elements in the first string
of the permutation.  Note that the $i$th hook 
has $u_i + 1$ horizontal boxes, where
$u_i$ is the number of unbarred elements in the first
string of the permutation.
See Figure~1.

\begin{theorem}
\label{theorem_bijection_fixed_point}
There exists an explicit bijection between the set of all fixed point free
permutations in the symmetric group on $n$ elements
and the set of all standard skew diagrams $\lambda/\mu$ having
$n$ boxes 
and hooks  of size greater than $1$.
\end{theorem}

\begin{proof}
We describe an algorithm to move between these two sets.
The idea is to first break
a cycle  
at the 
end of each of its descent runs
to form blocks.
Each of these blocks will  become a hook
in the resulting skew diagram.
The next step is to use the first element of each block
(for the first block, use the second element)
to determine which elements will be barred in a given block.
The third step is to reverse the order of the blocks.
The fact that the original first block contained the smallest element
in the given cycle will enable us to recover the 
complete cycle decomposition
of a permutation from its skew diagram in the general case
when a permutation has more than one cycle.

We first consider the case where $\pi = (\pi_1, \ldots, \pi_n) \in \Sym_n$ 
consists of a single
cycle of length $n$ with $\pi_1 = 1$, that is, the smallest
element of the set $\{\pi_1, \ldots, \pi_n\}$.

\begin{enumerate}

\item 
Identify the descents within the cycle.  
For each run of consecutive
descents,
say $[i,j] = i, i+1, \ldots, j$, break the permutation 
in front of the last descent in the run, 
that is,
the $(j-1)$st position
provided this 
does not create a first block having size one.

\item 
Suppose reading from left to right the first element in
the $i$th block is $m_{i_j}$, where 
the elements in this block have the linear order
$m_{i_1} < m_{i_2} < \cdots$.
(For the case of the first block, let $m_{1_j}$ be
the second element in this block when reading from
left to write
and where the block elements have
linear order
$m_{1,1} < m_{1,2} < \cdots$.)
Rewrite the elements in the block in 
decreasing order and place bars over
each of the last $j-1$ elements.

\item 
Reverse the order of the blocks, that is,
if $B_1 | B_2 | \cdots | B_k$ is the
original block decomposition, reverse this
to
$B_k | B_{k-1} | \cdots | B_1$.
Finally, remove the vertical block separators.

This yields the union of unsigned hooks, where a hook consists of the run
of unbarred elements followed by the run of barred elements.

\end{enumerate}

\begin{example} 
As an example, let $\pi = (135764928) \in \Sym_9$.  We have
\begin{eqnarray*}
\pi &\rightarrow& 1357 | 64 | 928\\
    &\rightarrow& 1357 | 46 | 298\\
    &\rightarrow& 753\bar{1} | 6 \bar{4} | 9 \bar{8}\bar{2}\\ 
    &\rightarrow& 9 \bar{8}\bar{2} 6 \bar{4} 7 5 3 \bar{1}  
\end{eqnarray*}
\end{example}

If a permutation consists of more than one cycle,
without loss of generality we may assume the permutation
is written in standard cycle notation,
that is,
each cycle is written so that it begins with the smallest
element in its cycle and the cycles are then ordered 
in increasing order by the smallest
element in each cycle.
Given such a permutation,
apply the algorithm to each individual cycle.
Concatenate the resulting barred words using the original
order of the cycles.

We can reverse this process beginning with a
standard 
unsigned skew diagram.
\begin{enumerate}

\item 
Given a standard unsigned skew diagram,
we will separate it into cycles based on the
  minimal element.  
Break the diagram after the hook containing the element $1$.
Next, break the diagram after the hook containing the smallest element
occurring to the right of the first break.
Then break after the hook containing the smallest element 
to the right of the second break.
Continue this process until there is a
break at the end of the diagram.
These breaks now correspond to 
individual cycles in the final permutation.

\item 
Within each of these breaks,
put parentheses around the elements of each hook
and reverse the order
of the hooks, that is, if
break $i$ has hooks
$h_{i,1} h_{i,2}\cdots h_{i,j}$ then
reverse these to
$h_{i,j} h_{i,j-1} \cdots h_{i,1}$.

\item 
In each parenthetical piece, remove the bars and reorder the 
elements by the
following rule.  
The now unbarred elements in each parenthesis can
be linearly ordered, say 
$m_{i_1} < m_{i_2} < \cdots < m_{i_k}$.  
If there were bars over $j$ numbers in this
piece,
reorder the elements as
$m_{i_1} m_{i_{j+1}} m_{i_2} \cdots m_{i_k}$ 
if $j \not= k$ 
and $m_{i_1} m_{i_k} m_{i_2} \cdots m_{i_{k-1}}$ if $j=k$.

\item 
Within each cycle, leave the vertical bars fixed for the moment
and switch the first two
numbers of all
the parenthetical pieces except the first piece which begins the cycle.  
Remove the inner parentheses and concatenate the pieces within
each vertically barred piece into one cycle.
\end{enumerate}

These processes we have described
are the inverse of each other.  Thus we 
have a bijection.
\end{proof}

\begin{example}
Let $8\bar{7}\bar{2} 6 \bar{1} 9\bar{5} 4
\bar{3}$ be a falling barred permutation.
The algorithm gives:
\begin{eqnarray*}
8\bar{7}\bar{2} 6 \bar{1} 9\bar{5} 4
\bar{3} &\rightarrow& 8\bar{7}\bar{2} 6 \bar{1}|
9\bar{5} 4\bar{3} |\\
&\rightarrow& (6\bar{1})(8 \bar{7} \bar{2}) | (4
\bar{3}) (9 \bar{5})\\
&\rightarrow& (16)(287)|(34)(59)\\
&\rightarrow& (16827)(3495)
\end{eqnarray*}
\end{example}

Let $F\subseteq [n-1]$ be the set of fixed points for a 
permutation $\pi \in \Sym_{n-1}$.  
We will build $n$ ordered pairs, $(F_i, \tau)$
where $i=1, \ldots, n$ and $\tau$ is a partial permutation on
$n-|F|-1$ elements from the set~$[n]$.
Set 
$$ 
F_i =  \left\{ \begin{array}{ll}
               F \cup \{i\} & \mbox{ if } i \notin F,\\
               F \cup \{n\} & \mbox{ if } i \in F,
               \end{array}
        \right.
$$
where $i = 1, \ldots, n$.
To define $\tau$, consider the partial permutation $\widehat{\pi}$
consisting of  the
cycles of
$\pi$ with sizes greater than $1$.  The
elements in these cycles can be linearly ordered as 
$m_{i_1} < m_{i_2} < \cdots < m_{i_{n-|F|-1}}$.  
The elements of $[n] - F_i$ also 
can be linearly ordered
as 
$l_{i_1} < \cdots < l_{i_{n-|F|-1}}$.  Define a map~$\Psi$ which
sends $m_{i_j} \mapsto l_{i_j}$.  Set $\tau=\Psi(\widehat{\pi})$.
Let $F_{\pi} = \{(F_i,\tau) :  i = 1, \ldots, n\}$
 so that $|F_{\pi}|=n$.

\begin{proposition}
There exists a bijection between 
$\{F_{\pi} : \: \pi \in \Sym_n\}$ and the
set of standard unsigned skew diagrams 
where each hook except the first  has size
greater than one.
\end{proposition}

\begin{proof}  
Given a permutation $\pi$
with  fixed point set $F$ and one ordered pair
$(F_i,\tau)$, we will define a map which sends $F_i$ to the first
hook of the diagram and which sends $\tau$ to the rest of the 
diagram.
To create the first part of the map, 
write the elements of $F_i$ in decreasing
order.  To place the bars, consider two cases.
\begin{enumerate}
\item 
If $i \notin F$ place bars over the element
$i$ and every element less than
  $i$.

\item 
If $i \in F$ we use the 
linear total order on $F$, say
$f_1 < \cdots <   f_{|F|}$.  
We have $i = f_j$ for some $j=1,\ldots, |F|$.
Place bars over 
  the smallest $j$ elements.
\end{enumerate}
This map can be reversed given the first piece of some unsigned 
skew diagram.

To determine the rest of the diagram, we use $\tau$, a partial permutation
on an $n-|F|-1$
element subset of $[n]$.  There is a bijection between all such
partial permutations and the set of fixed point free permutations
in $\Sym_{n-|F|-1}$.  Use the linear order on the elements of
$\tau$, that is, these elements can be written 
$m_{i_1} < \cdots < m_{i_k}$.  
Let $\Phi$ be a map between these two sets where
$\Phi(m_{i_j}) =j$.  Note
that because the partial permutation $\tau$
can be written as a product of cycles with no one-cycles, then
$\Phi(\tau)$ is also a fixed point free product of cycles.  Composing
$\Phi$ with the algorithm above, we can go from a partial permutation
$\tau$ to the rest of the diagram having hook sizes greater than 1.
\end{proof}

To prove 
Theorem~\ref{theorem_permanent}, we sign the first 
hook (which consists of the horizontal piece $0$ concatenated
with the vertical piece) in one way, that is, with all negative
signs, and then reorder the elements
in decreasing order.
For the remaining hooks,
we can sign these remaining elements in
$2^{n-|F|-1}$ ways and within each hook reorder 
them in a decreasing manner in one way.

As a corollary, we can slightly modify our proofs to give a bijective
proof of Jonsson's result (Theorem~\ref{theorem_Jonsson})
for the M\"obius function
of the Rees product of the Boolean algebra with the chain.

\vspace{2mm}
\begin{proof_special}{Theorem~\ref{theorem_Jonsson}}
It is enough to observe that Rees$(B_n, C_n)$ is isomorphic to the
upper order ideal generated by any atom of Rees$(\mathscr{C}_n,
C_{n+1})$.  Hence Rees$(B_n,C_n)$ inherits the $R$-labeling of
Rees$(\mathscr{C}_n, C_{n+1})$.  The maximal chains in $\Rees(B_n,C_n)$
are described by augmented barred permutations, that is,
permutations of the form
$\pi= \pi_1 \cdots \pi_n \pi_{n+1}$ with
$\pi_{n+1} = n+1$,
$|\pi| = |\pi_1| \cdots |\pi_n| \in \Sym_n$
(unlike before,
here  $|\pi_j|$ denotes removing any bar and negative
sign occurring in $\pi_j$),
$\pi_1$ not barred and each of the elements
$\pi_2, \ldots, \pi_n$ may be barred.  The falling chains correspond
to unsigned labeled skew diagrams
having hooks of size greater than or equal to $2$ which are augmented
at the end with a block containing the element $n+1$.
Theorem~\ref{theorem_bijection_fixed_point} now applies to prove the result.
\end{proof_special}

Shareshian and Wachs~\cite[Theorem~6.2]{Shareshian_Wachs_poset_homology}
have proved a dual version of 
Proposition~\ref{proposition_recursion}
where they instead work with 
a doubly-truncated  face lattice of the crosspolytope
$\mathscr{O}_n$.
To state their results we use 
$\Rees^{-}(P,Q)$ to indicate
the maximal element is removed from $P$ before taking the
Rees product of two graded posets $P$ and $Q$, that is, 
$
   \Rees^{-}(P,Q) = \Rees(P - \{\ho\},Q).
$

\begin{theorem}
[Shareshian--Wachs]
\label{theorem_crosspolytope}
For all $n$,
$$
   \dim \widetilde{H}_{n-1}(\Delta(\Rees^{-}(\mathscr{O}_n, C_n))) 
    = D_n^{\pm}.
$$
\end{theorem}
Shareshian and Wachs' original proof follows from the
Bj\"orner--Welker Theorem~\ref{theorem_Bjorner_Welker} and from the
fact that the reduced homology of a Cohen-Macaulay poset vanishes everywhere
except the top dimension, where the dimension is given by the
M\"obius function of the poset.
One can also give a bijective proof along the lines
of Theorem~\ref{theorem_bijection_fixed_point}
using the standard $R$-labeling of the cross-polytope.
For $q$-analogues of 
Theorems~\ref{theorem_Jonsson} and~\ref{theorem_crosspolytope},
see~\cite[Theorem 2.1.6 and Theorem 2.4.5]{Shareshian_Wachs_poset_homology}.

\section{A basis for the homology}
\label{section_basis}
\setcounter{equation}{0}

Let $P$ be a graded poset of rank $n$
with minimal element $\hz$ and
maximal element $\ho$. 
The {\em order complex} (or
{\em chain complex}) of $P$,
denoted $\Delta(P)$, is the simplicial complex 
with vertices given by the elements of
$P - \{\hz, \ho\}$
and $(i-1)$-dimensional faces are given by
chains of $i$ elements
$x_1 < x_2 < \cdots < x_i$
in the subposet
$P - \{\hz, \ho\}$.
See~\cite{Wachs_poset_topology} for further details.
In this section we consider  homological questions for the
order complex of the poset $\Rees({\mathscr C}_n, C_{n+1})$.
A similar analysis for 
the $d$-divisible partition lattice
was done by Wachs~\cite{Wachs}.

\begin{proposition}
The order complex $\Delta(\Rees({\mathscr C}_n, C_{n+1}))$ is
a Cohen-Macaulay complex and
has vanishing homology groups in every dimension except 
for the top dimension.
This is given by
$$
     \dim \widetilde{H}_n(\Delta(\Rees({\mathscr C}_n, C_{n+1})) 
     = n \cdot D_{n-1}^{\pm}.
$$
\end{proposition}
This follows by 
a result of Bj\"orner and Welker~\cite{Bjorner_Welker}
that 
the Rees product of any two Cohen-Macaulay posets is also
Cohen-Macaulay.
Furthermore, the absolute value of the
M\"obius function of the poset 
$\Rees({\mathscr C}_n, C_{n+1})$ gives the dimension
of the top homology group
of
$\Delta(\Rees({\mathscr C}_n, C_{n+1}))$.

We next give an
explicit basis for the homology
$\widetilde{H}_n(\Delta(\Rees({\mathscr C}_n, C_{n+1})),\Zzz)$ indexed 
by the falling augmented signed barred
permutations.
Recall that
 $\Fall_n$ denotes the set of falling augmented signed
barred permutations from $\barredsigned{n}$.
For each $\sigma \in \Fall_n$ we  define
a subposet $\Cbasis_{\sigma}$ of 
$\Rees({\mathscr C}_n, C_{n+1})$ as follows.
Let
$m_\sigma =  m_{\sigma,0} \prec m_{\sigma,1}
                          \prec \cdots \prec m_{\sigma,n}$
be the chain in
$\Rees({\mathscr C}_n, C_{n+1}) - \{\hz,\ho \}$ labeled by $\sigma \in \Fall_n$.
For example, for the
double augmented barred signed permutation
$\sigma = \sigma_0 \cdots \sigma_6
        = 0 \:\: -1 \:\: \overline{-3} \:\: 5
            \:\: \overline{2} \:\: \overline{-4}
            \:\: 6$,
we have
$     m_\sigma = (01001,1) \prec (*1001,1)
                          \prec (*1*01,2)
                          \prec (*1*0*,2)
                          \prec (***0*,3)
                          \prec (*****,4)$.

We define the elements of $\Cbasis_{\sigma}$ recursively. 
The rank $0$
elements of $\Cbasis_{\sigma}$ are of the form $(x,1)$,
where $x$ is a $0$-dimensional face of
the $n$-cube.
For $1 \leq i \leq n-1$, the
rank $i$ elements of $\Cbasis_{\sigma}$ are
of the form $(x,j)$,
where $x$ is an $i$-dimensional face of
the $n$-cube
and the second coordinate $j$ is determined according to the
following rules:
\begin{enumerate}
\item[$i$.]
If $\sigma_{i-1}$ is not barred, $\sigma_{i}$ is not barred,
and
$\sigma_{i+1}$ is either barred or unbarred,
then
$j = k$ where $(y,k)$ is any rank $i-1$ element of $\Cbasis_{\sigma}$.

\item[$ii$.]
If $\sigma_{i-1}$ is either barred or unbarred,
and both $\sigma_{i}$ and $\sigma_{i+1}$ are  barred, then
$j = k + 1$ where $(y,k)$ is any rank $i-1$ element of $\Cbasis_{\sigma}$.

\item[$iii$.]
If $\sigma_{i-1}$ is either barred or unbarred,
$\sigma_{i}$ is barred and
$\sigma_{i+1}$ is not barred, then
$j = k$ where $(y,k)$ is a rank $i-1$ element of $\Cbasis_{\sigma}$.
The exception to this rule
is for the $i$-dimensional
element
$x$ occurring in the chain $m_{\sigma}$,
that is,
$m_{\sigma,i} = (x,r)$.
In this case, $m_{\sigma,i}$ becomes  the element $(x,k+1)$ in $\Cbasis_{\sigma}$.

\item[$iv$.]
If $\sigma_{i-1}$ is barred,
$\sigma_{i}$ is not barred,
and $\sigma_{i+1}$ is either barred or unbarred,
then
$j = k + 1$ where $(y,k)$ is any rank $i-1$ element of $\Cbasis_{\sigma}$
different from
$m_{\sigma, i-1}$.
Notice that both
$m_{\sigma,i-1}$ and
$m_{\sigma,i}$ have the same second coordinate, namely $k+1$.

\end{enumerate}
Finally, there are two rank $n$ elements
$(* \cdots *, k)$ and $(* \cdots *, k+1)$,
where $k$ is the second coordinate of any
rank $n-1$ element of $\Cbasis_{\sigma}$.

Define $\widetilde{\mathscr{C}_{n}}$ to be the poset
$\mathscr{C}_{n}- \{\hz\} \cup \{\hat{1'}\}$,
that is,
the face lattice of the $n$-dimensional
cube with its minimal element removed
and adjoined with a second maximal
element $\hat{1'}$ which also covers all the coatoms in
$\mathscr{C}_n - \{\hz\}$.

\begin{theorem}
\label{theorem_suspension_barycentric_cube}
For $\sigma \in \Fall_n$,
the order complex
$\Delta(\Cbasis_{\sigma})$ is isomorphic to the 
suspension of the barycentric subdivison of the
boundary of the $n$-cube.
\end{theorem}
\begin{proof}
It is enough to show the posets 
$\Cbasis_{\sigma}$ and $\widetilde{\mathscr{C}_{n}}$ are isomorphic.
Define the ``forgetful'' map
$f: \Cbasis_{\sigma} \rightarrow \widetilde{\mathscr{C}_n}$
which sends an element
$(x,k) \in \Cbasis_{\sigma}$ to the element $x$
for elements of ranks $1$ through $n-1$ in $\Cbasis_{n-1}$.
For the two rank $n$ elements,
let
$f(* \cdots *, j_n) = \ho$
and
$f(* \cdots *, j_n + 1) = \hat{1'}$.
Clearly the map $f$ is a bijection from
the elements of $\Cbasis_{\sigma}$ to those of $\widetilde{\mathscr{C}_n}$.
Additionally, $f$ is order-preserving since for
$(y,k) \prec (x,j)$ in $\Cbasis_{\sigma}$, one has
$y \prec x$ in the cubical lattice $\mathscr{C}_n$

To define the inverse map $f^{-1}$,
one follows the described scheme
to determine the second coordinate as
above.  Note that for elements
$x$ and $y$
with $y \prec x$ in $\widetilde{\mathscr{C}_n}$ and
$\rho(y) < n$,
the inverse map
satisfies $f^{-1}(y) = (y,k) \prec f^{-1}(x) = (x,j)$
since $k \leq j$ by construction.  The two maximal elements of
$\widetilde{\mathscr{C}_n}$ are easily seen to be mapped to the two maximal
elements of $\Cbasis_{\sigma}$, so the bijection is order-preserving
as desired.
\end{proof}

\begin{corollary}
For $\sigma \in \Fall_n$,  the order complex $\Delta(\Cbasis_{\sigma})$
is homotopy equivalent
to the suspension of the
$(n-1)$-dimensional sphere $S^{n-1}$.
\end{corollary}
\begin{proof}
The order complex of
$\mathscr{C}_{n}$ is the barycentric
subdivision of the boundary of the $n$-cube.  The boundary of the
$n$-cube is homotopic to $S^{n-1}$.
The poset $\widetilde{\mathscr{C}_{n}}$ differs 
from $\mathscr{C}_n -\{ \hz, \ho \}$ by the addition
of two maximal elements $\ho$ and $\hat{1'}$.  
Therefore, the order complex of
$\widetilde{\mathscr{C}_{n}}$
is found from $\Delta(\mathscr{C}_n)$ by forming
two $(k+1)$-dimensional faces 
on the vertices
$V(\psi) \cup \{\ho\}$
and
$V(\psi) \cup \{\hat{1'}\}$,
where
$V(\psi)$ are the vertices of
a $k$-face $\psi$ in
$\Delta(\mathscr{C}_n)$.
This is a suspension
over the barycentric subdivision of the boundary of the $n$-cube which
is homotopic to the suspension of $S^{n-1}$.
Thus, by
Theorem~\ref{theorem_suspension_barycentric_cube}
we then have
$\Delta(\Cbasis_{\sigma}) \cong \Delta(\widetilde{\mathscr{C}_{n}})$ and we 
have proven the
corollary.
\end{proof}

The suspension of $S^{n-1}$ is homotopic to $S^n$, and as a result
$\Delta(\Cbasis_{\sigma})$ is a triangulation of the
$n$-sphere.
Let $\rho_{\sigma}$ denote a fundamental cycle of the spherical
complex $\Delta(\Cbasis_{\sigma})$.  
To show that the set 
$\{\rho_{\sigma} : \sigma \in \Fall_{n}\}$ forms a basis for 
$\widetilde{H}_n(\Delta(\Rees({\mathscr C}_n, C_{n+1})))$,
we first place a total order on $\Fall_n$.
Let $\sigma = 0 \sigma_1 \cdots \sigma_n \:n+1$
and
$\tau = 0 \tau_1 \cdots \tau_n \:n+1$
be two permutations from $\Fall_n$.
If the entries $\sigma_1, \ldots, \sigma_{i-1}$ and
$\tau_1, \ldots, \tau_{i-1}$ are unbarred,
$\sigma_i$ is
barred and $\tau_i$ is unbarred, then
we say $\sigma > \tau$.
Otherwise, if $\sigma$ and $\tau$ are barred and unbarred
at exactly the same
places and the permutation $\sigma$ without the bars is
lexicographically greater than
the permutation $\tau$ without its bars,
then we say $\sigma > \tau$.
We then have

\begin{lemma}
\label{lemma_chainorder}
If $m_\tau$ is a maximal chain in $\Cbasis_{\sigma}$
then $\tau \leq \sigma$.
\end{lemma}
\begin{proof}
Let $\sigma$ and $\tau$ be permutations in $\Fall_n$
with $\tau > \sigma$. 
We want to show $m_\tau$ is not a chain in $\Cbasis_\sigma$.
There are two cases to consider.

First suppose
$\sigma$ and $\tau$ are barred at precisely the same locations
and that
the unbarred permutation $\tau$ is
lexicographically greater than the unbarred permutation $\sigma$.
Let $i$ be the least index where
$\sigma_i$ is barred
and $\sigma_{i+1}$ is not barred. 
If such an $i$ does not
exist, then each permutation corresponds to a diagram
consisting of one hook
and as such has a labeling $-1 \cdots -n$,
implying $\sigma = \tau$, a contradiction.
So
we may assume such an $i$ satisfying
$1 \leq i < n$ exists.  We see the first $i$ elements
in the chain $m_\tau$ are elements in $\Cbasis_\sigma$.  However, the $i$th 
element
$m_{\tau, i}=(x,r_{\sigma,i}+1)$ where $x$ is the unique rank $i$ element in
$\mathscr{C}_n - \{\hz\}$
given by the unbarred permutation $\tau$ and
$r_{\sigma,i}$ is the number of bars over elements $\sigma_1, \sigma_2, 
\ldots,
\sigma_i$, will not be an element in
$\Cbasis_\sigma$.  (Note, the second coordinate $r_{\sigma,i} + 1$ in 
$m_{\sigma,i}$
is the same as the $i$th second coordinate $r_{\tau,i}+1$ in 
$m_{\tau,i}$ for all
$i$.)  We can see this by observing that the only element in
$\Cbasis_\sigma$ with first coordinate a rank $i$ element in
$\mathscr{C}_n - \{\hz\}$ and
with second coordinate $r_{\sigma,i} +1 = r_{\tau,i}+1$ corresponds to
the unique element given by the unbarred $\sigma$.  Thus, since the
unbarred $\sigma$ is not equal to the unbarred $\tau$, $m_\tau$ is not
a chain in $\Cbasis_\sigma$.

For the second case, suppose that
$\sigma_j$ is barred if and only if  $\tau_j$ is barred
for $j=1, \ldots, i-1$ and $\tau_i$ is barred while $\sigma_i$ is
not barred.
We claim the $i$th element
$m_{\tau, i}$ in $m_\tau$ is not an element of $\Cbasis_\sigma$.  Note the
second coordinate $r_{\tau,j} +1$ in $m_{\tau,j}$ is the same as the second
coordinate $r_{\sigma,j} +1$ in  $m_{\sigma,j}$
where $j = 1, \ldots, i-1$ because
the pattern of bars coincide
for the first $i-1$ terms in the
permutations. 
However, in $\Cbasis_\sigma$ all rank $i$ elements
$\mathscr{C}_n - \{\hz\}$
have second coordinate $r_{i-1}$. 
As there is no bar
over $\sigma_i$, the second coordinate does not increase. 
Since the element $\tau_i$
is barred, the element $(x,r_{\sigma,i-1}+1)$ is
an element in the chain $m_\tau$ but
not in the poset $\Cbasis_\sigma$.
\end{proof}

\begin{theorem}
The set $\{\rho_{\sigma} : \sigma \in \Fall_{n}\}$ forms a basis for
$\widetilde{H}_n(\Delta(\Rees({\mathscr C}_n, C_{n+1})))$ over $\Zzz$.
\end{theorem}
\begin{proof}
To show that
$\{\rho_{\sigma} : \sigma \in \Fall_{n}\}$
are linearly independent,
let
$\sum_{\sigma \in \Fall_{n}} a_{\sigma}\rho_{\sigma} =0$
where
$a_\sigma \in {\bf k}$. 
With respect to  the total order
we have described above, suppose $\tau$
is the greatest element of $\Fall_{n}$ for which $a_{\tau}\not= 0$.
We apply Lemma~\ref{lemma_chainorder}
to derive a contradiction.
We have for a maximal chain $m_r$ in $C_{\sigma}$ 
$$
     0 = \sum_{\sigma \in \Fall_{n}} a_{\sigma} \rho_{\sigma} |_{m_{\tau}}
       = \sum_{\sigma \in \Fall_{n}, \sigma \leq \tau} a_{\sigma}
                                   \rho_{\sigma} |_{m_{\tau}}
       = a_{\tau} \rho_{\tau} |_{m_{\tau}}
       = \pm a_{\tau},\\
$$
since
the fundamental cycle
evaluated at a facet has coefficient $\pm 1$.
However, this gives a contradiction.
Since the rank
of $\widetilde{H}_n(\Delta(\Rees({\mathscr C}_n, C_{n+1})))$ 
is equal to $|\Fall_{n}|$ , we have proven
the
basis result when {\bf k} is a field.

When ${\bf k} = \Zzz$,
linear independence of
$\{\rho_{\sigma} : \sigma \in \Fall_{n}\}$
implies this set is also linearly independent over the rationals
$\Qqq$ and hence that it spans
$\widetilde{H}_n(\Delta(\Rees({\mathscr C}_n, C_{n+1})))$ over $\Qqq$.
Let $\rho \in \widetilde{H}_n(\Delta(\Rees({\mathscr C}_n, C_{n+1})))$.  Then
$\rho = \sum_{\sigma \in \Fall_n} c_{\sigma} \rho_{\sigma}$
for $c_{\sigma} \in \Qqq$.
We will show $c_{\sigma} \in \Zzz$ for all
$\sigma \in \Fall_n$.
Suppose $\tau$ is the greatest element of $\Fall_n$ for which
$c_{\tau} \neq 0$.
Then by Lemma~\ref{lemma_chainorder}
$$
     \rho|_{m_{\tau}}
                      = \sum_{\sigma \in \Fall_n, \sigma \leq \tau}
                            c_{\sigma} \rho_{\sigma}|_{m_{\tau}}
                      = c_{\tau} \rho_{\tau}|_{m_{\tau}} 
                      = \pm c_{\tau}.
$$
Since $\rho \in \widetilde{H}_n(\Delta(\Rees({\mathscr C}_n, C_{n+1})))$,
we have $\rho|_{m_{\tau}} \in \Zzz$.
Thus
$c_{\tau} \in \Zzz$ and
$\rho - c_{\tau} \rho_{\tau} \in \widetilde{H}_n(\Delta(\Rees({\mathscr C}_n, C_{n+1})))$.
Repeat this argument for
$\rho - c_{\tau} \rho_{\tau}$
to conclude
$c_{\nu} \in \Zzz$
and
$\rho - c_{\tau} \rho_{\tau} - c_{\nu} \rho_{\nu} 
\in \widetilde{H}_n(\Delta(\Rees({\mathscr C}_n, C_{n+1})))$
for
$\nu$ the next to the last element in the total order
on $\Fall_n$ for which
$c_{\nu} \neq 0$.
Since there are finitely-many elements in $\Fall_n$,
we may conclude that
$c_{\sigma} \in \Zzz$ for all
$\sigma \in \Fall_n$.
Hence
$\{\rho_{\sigma} :  \sigma \in \Fall_{n}\}$
spans $\widetilde{H}_n(\Delta(\Rees({\mathscr C}_n, C_{n+1})))$ 
over $\Zzz$ and thus is a basis for
$\widetilde{H}_n(\Delta(\Rees({\mathscr C}_n, C_{n+1})))$ over $\Zzz$.
\end{proof}

\section{Representation over $\Sym_n$}
\setcounter{equation}{0}

In this section we develop
a  representation of
$\widetilde{H}_n(\Delta(\Rees({\mathscr C}_n, C_{n+1})))$ over the symmetric group.  
This can be done using a
set of skew Specht modules.

The homology of the order complex of $\Rees({\mathscr C}_n, C_{n+1})$
is an $\Sym_n$-module
in the following manner.  
A signed permutation
$\pi \in \barredsigned{n}$
corresponds to a labeled maximal chain of the poset $\Rees({\mathscr C}_n, C_{n+1})$.
A permutation $\tau \in \Sym_n$
acts on the chains of $\Rees({\mathscr C}_n, C_{n+1})$ by sending the maximal chain labeled with $\pi$
to the maximal chain
whose labels are $\tau \pi$.  Note
that under the action of $\tau$
the placement of the bars is fixed and the signs remain attached
to the same numbers.
This action induces an action
on the faces of $\Delta(\Rees({\mathscr C}_n, C_{n+1}))$ and even further on the homology group itself
whose basis is indexed by a subset of chains in 
$\Rees({\mathscr C}_n, C_{n+1})$.  We have
$\tau \rho_{\pi} = \rho_{\tau \pi}$
for any basis element 
$\rho_{\pi} \in \widetilde{H}_n(\Delta(\Rees({\mathscr C}_n, C_{n+1})))$.

\begin{theorem}
\label{theorem_representation}
There exists an $\Sym_n$-module isomorphism between 
$$
  \widetilde{H}_n(\Delta(\Rees({\mathscr C}_n, C_{n+1})))
     \mbox{  and  }
  \bigoplus 2^{n-|\lambda_1|}S^{\lambda},
$$ 
where the direct sum is over all
partitions $\lambda$ with each $\lambda_i$ shaped into hooks as
described in Section~\ref{section_basis} taken with
multiplicity $2^{n-|\lambda_1|}$.
\end{theorem}

To prove this result, we
will need some tools from combinatorial
representation theory.  
For more details and background information,
see~\cite{Sagan}.

Recall that two
tableaux $t_1$ and $t_2$
of shape $\lambda$ are {\em row equivalent},
written $t_1 \sim t_2$,
if the entries in each row of $t_1$ are a permutation
of the entries in the corresponding row of
$t_2$.
A {\em tabloid of shape~$\lambda$}
({\em   $\lambda$-tabloid} or {\em tabloid},
for short)
is then an equivalence class
$\{t\}$.
For a fixed partition
$\lambda$ we denote by $M^{\lambda}$ the $k$-vector space having
$\lambda$-tabloids as a basis.
In the usual way a permutation
$\sigma \in \Sym_n$ acts on a $\lambda$-tableau
by replacing each entry by its image under $\sigma$.
Thus
$\sigma$ acts on a $\lambda$-tabloid
$\{t\}$ by
$\sigma \{t\} = \{\sigma t\}$.
For a tableau
$t$ of shape $\lambda$, the {\em polytabloid} corresponding
to $t$ is
$$
  e_t = \sum_{\sigma \in C_t} {\rm sgn}(\sigma) \{ \sigma t \},
$$
where the sum is over all permutations belonging to
the column stabilizer
$C_t$ of $t$.

The {\em Specht module} $S^{\lambda}$ is the submodule
of $M^{\lambda}$ spanned by
the polytabloids $e_t$, where $t$ has shape~$\lambda$.
The Specht module $S^{\lambda}$ is an $\Sym_n$-module
in the following manner.  
A permutation $\tau \in \Sym_n$
acts linearly on the elements of $S^{\lambda}$ by permuting
the entries
of $t$, that is, $\tau e_t = e_{\tau t}$.

Specht modules were developed to
construct all irreducible representations
of the symmetric group over $\Ccc$.  We will use these
modules to give a representation of 
$\widetilde{H}_n(\Delta(\Rees({\mathscr C}_n, C_{n+1})))$ over
$\Sym_n$.

Recall that a tableau $t$ is said to be {\em standard}
if the entries are increasing in each row and column
of~$t$.
The following theorem is originally due to Young, though
not in this form.  It is also due to Specht.

\begin{theorem}
[Specht, Young]
The set
\begin{equation*}
\{e_t : t \hbox{ is a standard $\lambda$-tableau}\}
\end{equation*}
is a basis for $S^{\lambda}$.
\end{theorem}

With these definitions in
mind, we can begin the proof of Theorem~\ref{theorem_representation}.  
Let $\h = \lambda - \mu$ be a skew diagram
consisting of the union of $k$ hooks, as described in Section 4, where
the $i$th hook has size
$|\h_i|$.  We consider the case where
$\h=\h_1 \cdots \h_k$ is fixed.

Define a new set
$\Fall_{n}^{-} =\{-\sigma : \sigma= \sigma_1 \cdots \sigma_n \in \Fall_n\}$
where $-\sigma= -\sigma_1
-\sigma_2 \cdots -\sigma_n$.  It is easily noted there exists a
bijection between $\Fall_{n}^{-}$ and $\Fall_{n}$.  We use this bijection to
move between basis elements of 
$\widetilde{H}_n(\Delta(\Rees({\mathscr C}_n, C_{n+1})))$ which correspond to
decreasing
labeled skew shapes
and standard tableaux which have increasing labels.

Consider the usual unsigned Specht module $S^{\lambda}$
in the case $\lambda$ is composed of hooks of size at least two
and is augmented at the end with a block containing
the element $n+1$.  
It
is generated by polytabloids which are
indexed by
standard labelings of $\lambda$.  Define an $\Sym_n$-module homomorphism
\begin{equation*}
   \theta : S^{\lambda} \rightarrow \widetilde{H}_n(\Delta(\Rees({\mathscr C}_n, C_{n+1})))
\end{equation*}
where $e_{t} \mapsto \rho_{-\sigma}$
for  $t$
a standard $\lambda$-tableau and $\sigma
\in \Fall_{n}^{-}$ is found by writing the labels on $t$ from left to
right and by placing bars over all elements which
occur in  the rightmost
columns of a hook.

We wish to extend this map over ``signings'' of $S^{\lambda}$.  Given a
standard polytabloid $e_t \in S^{\lambda}$ where~$t$ is a standard
tableau, we sign the elements occurring in the
last $k-1$ hooks of $t$,
that is,
sign the labels on  
$\lambda_2,  \ldots, \lambda_k$.  This can
be written as a subset $A \subset [n]- \{\lambda_1\}$ where $A$
corresponds to the elements in $t$ labeled with a negative sign.  For
each $e_t$ there are
$\sum_{j=0}^{n-|\lambda_1|}{{n-|\lambda_1|}\choose j}$ such signings,
or equivalently, such subsets $A$.
We let $e_t^A$ denote the signing by $A$ of the polytabloid $e_t$.  
Using the binomial theorem, there is an
isomorphism
\begin{equation*}
{\bigoplus_{j=0}^{n-|\lambda_1|}}{{n-|\lambda_1|}\choose
    j} S^{\lambda} \cong
2^{n-|\lambda_1|}S^{\lambda}.
\end{equation*}
This is an $\mathbb{C}\Sym_n$-module
with action $\pi e_{t}^{A} = e_{\pi t}^{A}$ where $\pi \in \Sym_n$
permutes the labels of the tableau $t$.


To show each Specht module $S^{\lambda}$ occurs with multiplicity
$2^{n-|\lambda_1|}$ in the top homology group, we extend
the map $\theta$ to
$\theta:{\sum_{j=0}^{n-|\lambda_1|}}{{n-|\lambda_1|}\choose
    j} S^{\lambda} \rightarrow \widetilde{H}_n(\Delta(\Rees({\mathscr C}_n, C_{n+1})))$ where basis elements are
mapped by $e_{t}^{A}
\mapsto \rho_{-\sigma}$.  
The permutation $\sigma$ is found by attaching
negative signs to the labels in $t$ which are also in $A$.  Then the
labels in each hook
written in increasing order.  As before,
we form the permutation $\sigma$ by writing down the labels
reading from left to right
with bars placed over the rightmost element in every
row.  Note that when
$A$ is empty, we are back in the usual unsigned case.

Set $E^{\lambda} = \{e_{t}^A\}$ where $t$ ranges over all standard Young
tableaux of shape $\lambda$ and $A$ ranges over all subsets of 
$[n]- \{\lambda_1\}$.

\begin{proposition}\label{thetabijsets}
The map
\begin{equation*}
\theta: E^{\lambda} \longrightarrow \{\rho_{\sigma} | \sigma \in \Fall_n
  \hbox{ and } sh(\sigma)=\lambda\}
\end{equation*}
is a bijection.
\end{proposition}


\begin{proof}
Let $\theta'$ be a map from
$\{\rho_{\sigma} : \sigma \in \Fall_n \mbox{ and }
                              sh(\sigma)=\lambda\}$
to $E^{\lambda}$.  Given $\sigma \in \Fall_n$ with shape
$\lambda$, we will define $\theta'(\rho_{\sigma})=e_{t}^{A}$ such that
$\theta(e_{t}^{A})=\rho_{\sigma}$.

Set
$\theta'(\rho_{\sigma})=e_{t}^{A}$ by labeling $\lambda$ from left to right
with the elements of $-\sigma$.  Then
in each hook, rearrange the labels so the absolute value of these
labels is increasing.  Call this labeling $t'$.  
The subset $A$ is determined by
the negatively-labeled elements in $t'$ and $t$ is given by the
absolute value of $t'$.

One can check $\theta(\theta'(\rho_{\sigma}))=\rho_{\sigma}$ and
$\theta'(\theta(e_{t}^{A} ))=e_{t}^{A}$.
\end{proof}

Proposition~\ref{thetabijsets} can be extended by linearity to a
vector space isomorphism between the two spaces.

We sum over all possible partitions and signings of $\lambda$ to get a
bijection between basis elements of
${\sum_{j=0}^{n-|\lambda_1|}}{{n-|\lambda_1|}\choose     j} S^{\lambda}$ 
and the basis elements of 
$\widetilde{H}_n(\Delta(\Rees({\mathscr C}_n, C_{n+1})))$ to
conclude the following corollary.

\begin{corollary}
The map
\begin{equation*}
\theta:\{E^{\lambda}\}_{\lambda} \longrightarrow \{\rho_{\sigma}
|\sigma \in \Fall_n\}
\end{equation*}
is a bijection where $\lambda$ ranges over all skew diagrams which
are finite unions of hooks of size at least two augmented at the end
by a block containing the element $n+1$.
\end{corollary}

Again we extend by linearity to a vector space isomorphism between these two
spaces.    It is left to prove the module isomorphism properties in order to
prove Theorem~\ref{theorem_representation}.
First, we look at which elements of $\Sym_n$ fix
basis elements of $\sum_{j=0}^{n-|\lambda_1|}{{n-|\lambda_1|}\choose
    j} S^{\lambda}$ and 
$\widetilde{H}_n(\Delta(\Rees({\mathscr C}_n, C_{n+1})))$.

For a given tableau $t$, define
$S_t = S_{\lambda_1} \times \cdots \times S_{\lambda_k}$
where the $\lambda_i$ are subsets
of $[n]$ corresponding to the labels of the $i$th hook $t$.

\begin{claim}
The polytabloid
$e_t^A$ satisfies
     $e_{t}^{A}= e_{\pi t}^{A}$ for all $\pi \in S_t$.
\end{claim}
\begin{proof}
Given such a permutation $\pi \in S_t$,
it acts on $t$ by permuting labels
only within individual hooks of $t$.  If the labels within a row are
permuted, the polytabloid is fixed because each tabloid is a row
equivalence class.  If a labels within a column are permuted, $\pi$
acts as an element of the column stabilizer
$C_t$.  For such an element $\pi$
we have $e_t=e_{\pi t}$.
Lastly, if an element in a column is moved out of its column but
within its 
row because of the equivalence class, we can rewrite the tabloid with
that element occurring at the end of the row, 
leaving $\pi$ to act as an element
of $C_t$.
\end{proof}

\begin{claim}
The fundamental cycle $\rho_{\sigma}$ satisfies
$\rho_{\sigma} =\rho_{\pi \sigma}$ for all $\pi \in S_t$.
\end{claim}
\begin{proof}
It is enough to show 
the posets $\Cbasis_\sigma$ and
$\Cbasis_{\pi   \sigma}$ 
are isomorphic
to prove the equality of the fundamental cycles of their
order complexes $\rho_\sigma$ and $\rho_{\pi \sigma}$.  
The elements of the posets $\Cbasis_\sigma$
and $\Cbasis_{\pi \sigma}$ have bars in the same places and negative signs
with the same numbers, so it is left to consider the ranks in the poset
where one element of rank $i$ for some $i$ has a different second
coordinate from all other elements of that rank.  
If the set of
ranks with this property is the
same in $\Cbasis_\sigma$ and $\Cbasis_{\pi \sigma}$, the two posets are 
isomorphic.
In $\sigma$ or $\pi \sigma$ an element having
rank $i$ must correspond to a label
at the end of a
piece $j$ for some $j$.  In $\Cbasis_\sigma$, this element will
have stars in positions corresponding to labels in the first $j$
places in the permutation.  This is the same in $\Cbasis_{\pi \sigma}$
because $\pi$ only permutes elements within individual pieces.  The
fixed negative signs assure the non-starred elements are the same in
both.  Thus, we have $\Cbasis_\sigma=\Cbasis_{\pi \sigma}$.
\end{proof}

We now prove Theorem~\ref{theorem_representation}.
For $\pi \in S_t$, we have
$\pi \theta (e_{t}^{A})=\theta(\pi e_{t}^{A})$.  That is,
\begin{eqnarray*}
\pi \theta(e_{t}^{A})=\pi \rho_{\sigma}
=\rho_{\pi \sigma} =\rho_{\sigma} = \theta(e_{t}^{A})= \theta (e_{\pi
  t}^{A}) = \theta(\pi e_{t}^{A})
\end{eqnarray*}
It is left to show this relationship holds for $\tau \in S_n -
S_t$.  In fact, it is enough to show $\theta(\tau e_{t}^{A})=\pi \tau
\rho_{\sigma}$
for some $\pi \in S_t$.

Consider $\theta( \tau e_{t}^{A})$ and $\tau \rho_{-\sigma}$ for some
$\tau \in S_n -
S_t$ and some $e_{t}^{A}$ such that
$\theta(e_{t}^{A})=\rho_{\sigma}$.  
The permutation~$\tau$ acts on $t$ by permuting
the labels.  
The polytabloid
$e_{\tau t}^{A}$ is a sum of tabloids under action by
the column stabilizer $C_t$.  Hence, we are only concerned with
cycles of $\tau$ which move labels from one hook of $t$ to another
hook of $t$.  Let $\theta$ take $e_{\tau t}^{A}$ onto
$\rho_{-\widehat{\sigma}}$.
We know $\widehat{\sigma}$ is found by
attaching the signs from $A$ to $t$ and
reordering so each piece is decreasing, and $\tau$ acts
on $\sigma$ also by permuting the labels.  There is no guarantee that
$\tau \sigma$ will have hooks
each of which having  labels in decreasing order.
However, we can find a
permutation
$\pi \in S_t$ such that $\pi \tau \sigma$ will
have hooks whose labels are in decreasing order.  Since there is only
one way to write a set of integers in decreasing order, it is left to
show the labels on each hook of $\widehat{\sigma}$ are the same as the
labels on the corresponding hook of $\tau \sigma$.  (Hooks of
$\sigma \in \barredsigned{n}$ correspond to the hooks in the
$\lambda$ associated with $\sigma$.)  If label $l$ is in a different
hook in $\widehat{\sigma}$ than
in $\tau \sigma$, then  $\tau t$ mapped $l$ to a different hook than
  $\tau \sigma$.  This is a contradiction because labels in $t$ are in
  the same corresponding hooks as labels in $\sigma$.  Hence,
  $\widehat{\sigma} = \pi \tau \sigma = \tau \sigma$ and $\theta(\tau
 e_{t}^{A}) = \rho_{-\tau \sigma}$.

The isomorphism
${\bigoplus_{j=0}^{n-|\lambda_1|}}{{n-|\lambda_1|}\choose
    j} S^{\lambda} \cong
2^{n-|\lambda_1|}S^{\lambda}$ induces the desired  module isomorphism
$2^{n-|\lambda|}  S^{\lambda} \cong \widetilde{H}_n(\Delta(\Rees({\mathscr C}_n, C_{n+1})))$.  Thus, we have
    proved Theorem~$\ref{theorem_representation}$.

\section{Concluding remarks}
\setcounter{equation}{0}

What poset $P$ would have its M\"obius function related to
the permanent
of a matric having $s$'s occur on the diagonal
and $r$'s in the off-diagonal entries?
The case when $s = r-1$ is the
Rees product of the
$r$-cubical lattice with the chain,
the case 
$(r,s) = (1,0)$ corresponds to the Rees product of the Boolean algebra
with the chain, and 
$(r,s) = (2,1)$ to the Rees product of the 
cubical lattice
with the chain.

The derangement numbers occur as the local $h$-vector
of the barycentric subdivision of the
$n$-simplex~\cite{Stanley_local_h}.
Is there a relation between the local $h$-vector
and the Rees product?

\section{Acknowledgements}
The authors would like to thank 
Richard Ehrenborg for
his comments on an earlier version of this paper and
Bruce Sagan for some historical background related to
representation theory.

%
%
%

\newcommand{\journal}[6]{{\sc #1,} #2, {\it #3} {\bf #4} (#5), #6.}
\newcommand{\book}[4]{{\sc #1,} ``#2,'' #3, #4.}
\newcommand{\bookf}[5]{{\sc #1,} ``#2,'' #3, #4, #5.}
\newcommand{\thesis}[4]{{\sc #1,} ``#2,'' Doctoral dissertation, #3, #4.}
\newcommand{\springer}[4]{{\sc #1,} ``#2,'' Lecture Notes in Math.,
                          Vol.\ #3, Springer-Verlag, Berlin, #4.}
\newcommand{\preprint}[3]{{\sc #1,} #2, preprint #3.}
\newcommand{\progress}[2]{{\sc #1,} #2, work in progress.}
\newcommand{\archive}[3]{{\sc #1,} #2, {\bf #3}.}
\newcommand{\unpublished}[1]{{\sc #1,} unpublished.}
\newcommand{\unpublisheddate}[2]{{\sc #1,} unpublished #2.}
\newcommand{\preparation}[2]{{\sc #1,} #2, in preparation.}
\newcommand{\appear}[3]{{\sc #1,} #2, to appear in {\it #3}.}
\newcommand{\submitted}[4]{{\sc #1,} #2, submitted to {\it #3}, #4.}
\newcommand{\AdvancesinMathematics}{Adv.\ Math.}
\newcommand{\DiscreteComputationalGeometry}{Discrete Comput.\ Geom.}
\newcommand{\DiscreteMath}{Discrete Math.}
\newcommand{\EuropeanJournalofCombinatorics}{European J.\ Combin.}
\newcommand{\JCTA}{J.\ Combin.\ Theory Ser.\ A}
\newcommand{\JCTB}{J.\ Combin.\ Theory Ser.\ B}
\newcommand{\JournalofAlgebraicCombinatorics}{J.\ Algebraic Combin.}
\newcommand{\communication}[1]{{\sc #1,} personal communication.}

\newcommand{\collection}[9]{{\sc #1,} #2, 
           in {\it #3} (#4), {\it #5},
           #6, #7, #8, #9.}


\vspace{.2in}

 \noindent
{\it
  Patricia Muldoon Brown,                        
  Department of  of Mathematics,                     
  Armstrong Atlantic State University,\\                        
  Savannah, GA            31419,
  {\tt patricia.brown@armstrong.edu}
\\[2 mm]
  Margaret A.\ Readdy,                          
  Department of Mathematics,                     
  University of Kentucky,                        
  Lexington, KY 40506,                      
  {\tt readdy@ms.uky.edu}
}

\end{document}